\documentclass[reqno]{amsart}
\usepackage{amssymb}
\usepackage{amsmath}
\usepackage{amsfonts}
\usepackage{amsxtra}
\usepackage{xcolor}
\usepackage{lipsum}

\newtheorem{theorem}{Theorem}[section]
\newtheorem{lemma}[theorem]{Lemma}
\newtheorem{proposition}[theorem]{Proposition}
\newtheorem{corollary}[theorem]{Corollary}

\theoremstyle{definition}

\newtheorem{example}[theorem]{Example}

\theoremstyle{remark}

\makeatletter
\newcommand{\cp}{\mathop{\operator@font cp}}
\newcommand{\range}{\mathop{\operator@font range}}
\newcommand{\rank}{\mathop{\operator@font rank}}
\newcommand{\dom}{\mathop{\operator@font dom}}
\newcommand{\Real}{\mathop{\operator@font Re}}
\newcommand{\cont}{\mathop{\operator@font Cont_w}}
\newcommand{\alg}{\mathop{\operator@font Alg\mathcal{N}}}
\newcommand{\tr}{\mathop{\operator@font tr}}
\newcommand{\sgn}{\mathop{\operator@font sgn}}
\newcommand{\kn}{\mathop{\operator@font \mathcal{K}(\mathcal{N})}}
\newcommand{\scr}{\mathop{\operator@font C_0(X)\times_{\phi}\mathbb{Z}_+}}
\newcommand{\rad}{\mathop{\operator@font Rad}}
\newcommand{\hrad}{\mathop{\operator@font HRad}}
\newcommand{\D}{\mathrm D}
\newcommand{\hc}{\mathop{\operator@font \mathcal{R}_{hc}(\mathcal{A})}}

\makeatother

\begin{document}

\title[Compact multiplication operators ]{Compact multiplication operators  on semicrossed products}

\author{G. Andreolas}
\author{ M. Anoussis}
\author{C. Magiatis}
\address{Department of Mathematics, University of the Aegean, 832\,00
Karlovassi, Samos, Greece}
\email{gandreolas@aegean.gr}
\address{Department of Mathematics, University of the Aegean, 832\,00
Karlovassi, Samos, Greece}
\email{mano@aegean.gr}
\address{Department of Mathematics, University of the Aegean, 832\,00
Karlovassi, Samos, Greece}
\email{chmagiatis@aegean.gr}

\subjclass[2010]{Primary 47L65; Secondary 47B07.}
\keywords{Semicrossed products, non-selfadjoint Operator Algebras, multiplication operators, compact elements, recurrent points, wandering points, equicontinuity.}

\begin{abstract}
We characterize  the compact multiplication operators on a semicrossed product $C_0(X)\times_\phi\mathbb Z_+$ 
 in terms 
of the corresponding dynamical system.
We also 
characterize the compact elements of this algebra  and determine the ideal they 
 generate.

\end{abstract}

\maketitle

\section{introduction}

Let $\mathcal{A}$ be a Banach algebra and  $a,b\in\mathcal A$. 
The map $M_{a,b}:\mathcal A\rightarrow\mathcal A$ given by $M_{a,b}(x)=axb$  is  called  a \emph{multiplication operator}. 
 Properties of compact
multiplication operators
have been investigated since 1964 when Vala published his work ``On compact sets of compact operators'' \cite{vala64}. Let
$\mathcal{X}$ be a normed space and $\mathcal{B}(\mathcal{X})$ the algebra of all bounded linear maps from $\mathcal{X}$ into
$\mathcal{X}$. Vala proved that a non-zero multiplication operator
$M_{a,b}:\mathcal{B}(\mathcal{X})\rightarrow\mathcal{B}(\mathcal{X})$
is
compact if and only if the operators $a, b \in\mathcal{B}(\mathcal{X})$  are both
compact.
Also, in \cite{vala67} Vala defines an element $a$ of a normed algebra to be \emph{compact} 
if the mapping $x\mapsto axa$ is compact.
 This concept enabled  the study of compactness properties of  elements of abstract normed algebras.
 Ylinen in \cite{1972} studied compact elements  for abstract C*-algebras and 
  showed that  $a$ is a compact element of a $C^*$-algebra $\mathcal{A}$  if and only if there  
 exists an isometric $*$-representation $\pi$ of   $\mathcal{A}$ on a Hilbert
space $\mathcal H$ such
that the operator $\pi(a)$ is compact. 

Compactness questions  have also been considered in the more general framework 
of elementary operators.  A map
$\Phi:\mathcal{A}\rightarrow\mathcal{A}$, where $\mathcal A$ is a Banach algebra, is called \textit{elementary} if $\Phi=\sum_{i=1}^m M_{a_i,b_i}$ for some
$a_i,b_i\in\mathcal{A}$, $i=1,\ldots,m$. Fong and Sourour showed that an elementary operator
$\Phi:\mathcal{B}(\mathcal H)\rightarrow\mathcal{B}(\mathcal H)$, where $\mathcal{B}(\mathcal H)$ is the algebra of bounded linear operators
on a Hilbert space $\mathcal H$, is compact if and only if there exist compact operators
$c_i,d_i\in\mathcal{B}(\mathcal H)$, $i=1,\ldots,m$ such that $\Phi=\sum_{i=1}^m M_{c_i,d_i}$ \cite{fs}. This result was expanded by
Mathieu on prime C*-algebras \cite{m} and later on general C*-algebras by Timoney \cite{tim}. 

 Akemann and Wright 
\cite{ake} characterized the weakly compact multiplication operators on  $\mathcal{B}(\mathcal H)$,
where  $\mathcal H$ is a Hilbert space.
Saksman and  Tylli \cite{st, st1}
and 
Johnson and Schechtman   \cite{js}  studied weak compactness of multiplication operators in a Banach space setting.

Moreover, strictly singular multiplication operators are studied by 
Lindstr\"{o}m, Saksman and Tylli  \cite{lst} and 
Mathieu and  Tradacete  \cite{mt}.

Compactness properties of multiplication operators on nest algebras, a class of non selfadjoint operator algebras,  
are studied by  Andreolas and  Anoussis 
in \cite{anan}.
In particular they characterized the compact multiplication operators, the compact elements and the ideal  generated by the compact
elements.

In the present paper we study multiplication operators on a  semicrossed product $C_0(X)\times_\phi\mathbb Z_+$ where $X$ 
is a locally  compact metrizable space, and $\phi: X\rightarrow X$ a homeomorhism.
We  characterize the compact 
multiplication operators  in terms of 
the corresponding dynamical system. As a consequence, we obtain a characterization of the compact
elements of the semicrossed product.
We also characterize the ideal generated by the compact elements. 

We would like to note that the equicontinuity condition appearing in the characterization of 
the compact multiplication operators on the semicrossed product, follows from the other conditions 
if $X$ is discrete 
or has no isolated points. However,   
in the general case this does not hold and thus   
 the proof is considerably more elaborated.

\section{Compact multiplication operators on semicrossed products}

Throughout this paper, $X$ will be a locally compact metrisable space and $\phi:X\rightarrow X$ a homeomorphism. The pair $(X, \phi)$ 
is called a dynamical system.
 An action 
of $\mathbb{Z}_+$ on $C_0(X)$ by isometric $*$-automorphisms $\alpha_n$, $n\in\mathbb{Z}_+$ is obtained by defining 
$\alpha_n(f)=f\circ\phi^n$.
We write the elements of the Banach space $\ell^1(\mathbb 
Z_+,C_0(X))$ as formal series $A=\sum_{n\in\mathbb Z_+}U^nf_n$ with the norm given by 
$\|A\|_1=\sum\|f_n\|_{C_0(X)}$.  The multiplication on $\ell^1(\mathbb Z_+,C_0(X))$ is 
defined by setting
$$U^nfU^mg=U^{n+m}(\alpha^m(f)g)$$
and extending by linearity and continuity. With this multiplication, $\ell^1(\mathbb Z_+,C_0(X))$ is a Banach algebra.

The Banach algebra $\ell^1(\mathbb Z_+,C_0(X))$ can be faithfully represented as a (concrete) operator algebra on a 
Hilbert space. This is achieved by assuming a faithful action of $C_0(X)$ on a Hilbert space 
$\mathcal{H}_0$. Then, we can define a faithful contractive representation $\pi$ of $\ell_1(\mathbb Z_+,C_0(X))$ on the Hilbert 
space $\mathcal H=\mathcal{H}_0\otimes \ell^2(\mathbb Z_+)$ by defining $\pi(U^nf)$ as
$$\pi(U^nf)(\xi\otimes e_k)=\alpha^k(f)\xi\otimes e_{k+n}.$$
The \emph{semicrossed product} $C_0(X)\times_{\phi}\mathbb Z_+$ is the closure of the image of 
$\ell^1(\mathbb Z_+,C_0(X))$ in $\mathcal{B(H)}$ in the representation just defined, where $\mathcal{B(H)}$ is the algebra of bounded linear operators on $\mathcal{H}$. We will denote the semicrossed product $C_0(X)\times_{\phi}\mathbb Z_+$ by $\mathcal{A}$ and an element $\pi(U^nf)$ of $\mathcal{A}$ by $U^nf$ to simplify the notation. The closed unit ball of $\mathcal{A}$ will be denoted by $\mathcal{A}_1$. We refer to \cite{dfk} and \cite{dkm}, for more information about the semicrossed product.

For $A=\sum_{n\in\mathbb Z_+}U^nf_n\in \ell^1(\mathbb Z_+,C_0(X))$ we call $f_n\equiv 
E_n(A)$ the $n$th \emph{Fourier coefficient} of $A$. The maps $E_n:\ell^1(\mathbb N_+,C_0(X))\rightarrow C_0(X)$ are contractive in the (operator) norm of $\mathcal A$, and therefore they extend to
contractions $E_n:\mathcal A \rightarrow C_0 (X)$.
Let $A\in\mathcal A$. If the set $\{m\in\mathbb Z_+:E_m(A)\neq0\}$ is finite, then $A$ is called a \emph{polynomial}. If there exists a unique $m\in\mathbb Z_+$, such that $E_m(A)\neq0$, then $A$ is called \emph{monomial}.


Let $(X,\phi)$ be a dynamical system. Then, a point $x\in X$ is called \textit{recurrent} if 
there exists a strictly increasing sequence $(n_k)_{k\in\mathbb{N}}\subseteq \mathbb{N}$, such 
that $\lim_{k\rightarrow\infty}\phi^{n_k}(x)=x$. The set of the recurrent points of $(X,\phi)$ will be denoted by $X_r$.
We will denote by 
 $X_i$ the set of the isolated points of $X$, by
 $X_a$ the set of the accumulation points of $X$ and we set
 $X_{a,i}=\{x\in X_{a}:\exists (x_j)\subseteq X_i,\ \lim_{j\rightarrow\infty}x_j= x\}$. If $f\in C_0(X)$, we set
 $\mathrm \D(f)=\{x\in X:|f(x)|>0\}$.


\begin{lemma}\label{cmo1}
Let $M_{A,B}:\mathcal A\rightarrow \mathcal A$ be a 
compact multiplication operator, where  $A,B\in\mathcal A_1$ and  $E_m(A)=f_m$, $E_m(B)=g_m$, 
for all $m\in\mathbb Z_+$. Then, $(f_m\circ\phi^{n+l}g_n)(X_a)=\{0\}$, for all $m, n, l \in\mathbb{Z_+}$.
\end{lemma}
\begin{proof}
We suppose that there exist $m, n, l\in\mathbb Z_+$ such that
$\left(f_{m}\circ\phi^{n+l}g_{n}\right)(X_a)\neq \{0\}.
$ We define the following indices.
\begin{eqnarray*}
n_0 & = & \min\left\{n\in\mathbb Z_+|\;\exists\; m,l\in\mathbb Z_+:\left(f_m\circ\phi^{n+l}g_n\right)(X_a)\neq \{0\}\right\},\\
m_0 & = & \min\left\{m\in\mathbb Z_+|\;\exists\; l\in\mathbb Z_+:\left(f_m\circ\phi^{n_0+l}g_{n_0}\right)(X_a)\neq \{0\}\right\},\\
l_0 & = & \min\left\{l\in\mathbb Z_+|\left(f_{m_0}\circ\phi^{n_0+l}g_{n_0}\right)(X_a)\neq \{0\}\right\}.
\end{eqnarray*}
Let $x_0\in X_a$, such that $\left(f_{m_0}\circ\phi^{n_0+l_0}g_{n_0}\right)(x_0)\neq 0$.
Then, there exist an $\epsilon>0$ and an open neighborhood $U_0$ of $x_0$ such that $\left|\left(f_{m_0}\circ\phi^{n_0+l_0}g_{n_0}\right)(x)\right|>2\epsilon$, for all $x\in U_0$. Now, we consider the quantity
\begin{eqnarray*}
\bar g=\inf_{x\in U_0}\{|g_{n_0}(x)|\}>0.
\end{eqnarray*}
 We consider  the function $J:X\rightarrow \mathbb{C}$, defined as follows.
\begin{eqnarray*}
J(x)=\left\{\begin{matrix} 0& \hspace{-1em}, & \mathrm {if}\;\; m_0+n_0=0 \\ 
\displaystyle \sum_{n=0,n\neq n_0}^{n_0+m_0}\left|(f_{m_0+n_0-n}\circ\phi^{n+l_0}g_ng_{n_0})(x)\right| & \hspace{-1em}, &\mathrm{if}\;\;  m_0+n_0>0. \end{matrix} \right.
\end{eqnarray*}
If $m_0+n_0>0$, we claim that $J(x_0) =  0$.
Indeed, if  $n<n_0$, then it follows from the definition of $n_0$ that $(f_{m_0+n_0-n}\circ\phi^{n+l_0}g_n)(X_a)=\{0\}$. Otherwise, 
if $n>n_0$ and $n\leq m_0+n_0$, it follows from the definition of $m_0$ that $(f_{m_0+n_0-n}\circ\phi^{l_0+n}g_{n_0})(X_a)=\{0\}$, since $m_0+n_0-n<m_0$. Therefore, there exists an open neighborhood $V_0$ of $x_0$, such that
$J(x) < \epsilon\bar g$, for all $x\in V_0,$ by the continuity of $J$.

If $m_0+n_0>0$, we set $W_0=U_0\cap V_0$, otherwise, we set $W_0=U_0$.
Since $x_0\in W_0\cap X_a$, there exist a sequence of points $\{x_i\}_{i\in\mathbb{N}}\subseteq W_0$, 
a sequence of open subsets $\{W_i\}_{i\in\mathbb{N}}\subset W_0$ with $x_i\in W_i$ and $W_i\cap W_j=\emptyset$, 
for $i\neq j$ and a sequence of norm one 
functions $\{h_i\}_{i\in\mathbb{N}}\subseteq C_{0}(X)$ with $\D(h_i)\subseteq W_i$ 
and $h_i(x_i)=1$, for all $i\in\mathbb{N}$.

To complete the proof, we consider the sequence $\{U^{l_0}h_i\circ\phi^{-n_0}\}_{i\in\mathbb{N}}$ and we will prove that the sequence
$\{M_{A, B}(U^{l_0}h_i\circ\phi^{-n_0})\}_{i\in\mathbb{N}}$ has no convergent subsequence.
We estimate the quantity $\|M_{A,B}\left(U^{l_0}h_u\circ\phi^{-n_0}\right)-M_{A,B}\left(U^{l_0}h_v\circ\phi^{-n_0}\right)\|_{\mathcal A}$, for $u,v\in\mathbb{N}$, $u\neq v$.
 
\begin{eqnarray*}
\left\|M_{A,B}\left(U^{l_0}h_u\circ\phi^{-n_0}\right)-M_{A,B}\left(U^{l_0}h_v\circ\phi^{-n_0}\right)\right\|_{\mathcal A} & \ge & \\ 
\left\|E_{m_0+l_0+n_0}\left(M_{A,B}\left(U^{l_0}h_u\circ\phi^{-n_0}-U^{l_0}h_v\circ\phi^{-n_0}\right)\right)\right\|_{C_0(X)} & = &\\ 
\left\|\sum_{n=0}^{m_0+n_0}f_{m_0+n_0-n}\circ\phi^{n+l_0}g_n(h_u\circ\phi^{n-n_0}-h_v\circ\phi^{n-n_0})\right\|_{C_0(X)}&\ge&\\
\left\|\sum_{n=0}^{m_0+n_0}f_{m_0+n_0-n}\circ\phi^{n+l_0}g_ng_{n_0}(h_u\circ\phi^{n-n_0}-h_v\circ\phi^{n-n_0})\right\|_{C_0(X)} & \ge &\\
\left|\sum_{n=0}^{m_0+n_0}\left(f_{m_0+n_0-n}\circ\phi^{n+l_0}g_ng_{n_0}(h_u\circ\phi^{n-n_0}-h_v\circ\phi^{n-n_0})\right)(x_u)\right| & \ge &\\
\left|\left(f_{m_0}\circ\phi^{n_0+l_0}g^2_{n_0}\right)(x_u)\right|&  &\\
-\sum_{n=0,n\neq n_0}^{m_0+n_0}\left|\left(f_{m_0+n_0-n}\circ\phi^{n+l_0}g_ng_{n_0}(h_u\circ\phi^{n-n_0}-h_v\circ\phi^{n-n_0})\right)(x_u)\right|.
\end{eqnarray*}
We note that $|(h_u\circ\phi^{n-n_0}-h_v\circ\phi^{n-n_0})(x_u)|\leq 1$, since  $\D(h_u)\cap \D(h_v)=\emptyset$. Therefore, we obtain
\begin{eqnarray*}
\left\|M_{A,B}\left(U^{l_0}h_u\circ\phi^{-n_0}\right)-M_{A,B}\left(U^{l_0}h_v\circ\phi^{-n_0}\right)\right\|_{\mathcal A} & \ge & \\ 
\left|\left(f_{m_0}\circ\phi^{n_0+l_0}g^2_{n_0}\right)(x_u)\right|-\sum_{n=0,n\neq n_0}^{m_0+n_0}\left|(f_{m_0+n_0-n}\circ\phi^{n+l_0}g_ng_{n_0})(x_u)\right| & = &\\
\left|\left(f_{m_0}\circ\phi^{n_0+l_0}g^2_{n_0}\right)(x_u)\right|-J(x_u) & > & \epsilon\bar g,
\end{eqnarray*}
which concludes the proof.
\end{proof}



\begin{lemma}\label{cmo2}
Let $M_{A,B}:\mathcal A\rightarrow \mathcal A$ be a compact multiplication operator, where  $A,B\in \mathcal A_1$ and  $E_m(A)=f_m$, $E_m(B)=g_m$, for all $m\in\mathbb Z_+$.
Then, $\lim_{l\rightarrow\infty}(f_m\circ\phi^{n+l}g_n)(x)= 0$, for all  $m,n\in\mathbb Z_+$ and $x\in X_i$.
\end{lemma}
\begin{proof}
We suppose that there exist $n,m\in\mathbb Z_+$ and $x\in X_i$, such that $\lim_{l\rightarrow\infty}(f_m\circ\phi^{n+l}g_n)(x)\neq 0$. We define the following indices.
\begin{eqnarray*}
n_0 & = & \min\left\{n\in\mathbb Z_+ | \exists m\in\mathbb Z_+,\ x\in X_i:\lim_{l\rightarrow\infty}(f_m\circ\phi^{n+l}g_n)(x)\neq 0 \right\},\\
m_0 & = & \min\left\{m\in\mathbb Z_+ | \exists x\in X_i:\lim_{l\rightarrow\infty}(f_m\circ\phi^{n_0+l}g_{n_0})(x)\neq 0 \right\}.
\end{eqnarray*}
Then, there exist an element $x_0\in X_i$, an $\epsilon>0$ and a strictly increasing sequence $\{l_i\}_{i\in\mathbb{N}}\subset\mathbb Z_+$, such that
$\left|(f_{m_0}\circ\phi^{n_0+l_i}g_{n_0})(x_{0})\right| > 2\epsilon.$ We consider the sequence $\{J(l)\}_{l\in\mathbb{Z}_+}$,
\begin{eqnarray*}
J(l)=\left\{\begin{matrix} 0& \hspace{-1em}, &\mathrm{if} \;\; m_0+n_0=0 \\
\displaystyle \sum_{n=0,n\neq n_0}^{n_0+m_0}\left|(f_{m_0+n_0-n}\circ\phi^{n+l}g_ng_{n_0})(x_0)\right| & \hspace{-1em}, &\mathrm{if} \;\; m_0+n_0>0. \end{matrix} \right.
\end{eqnarray*}
If $m_0+n_0>0$, we claim that there exists an $L\in\mathbb N$, such that: 
\begin{eqnarray*}
J(l)=\sum_{n=0 ,n\neq n_0}^{n_0+m_0}\left|(f_{m_0+n_0-n}\circ\phi^{n+l}g_ng_{n_0})(x_{0})\right| <  \epsilon g_{n_0}(x_0), & \forall\; l>L.
\end{eqnarray*}
Indeed, if $n<n_0$, it follows from the definition of $n_0$, that $\lim_{l\rightarrow\infty}(f_{m_0+n_0-n}\circ\phi^{n+l}g_n)(x_{0})= 0$.
On the other hand, if $n>n_0$ and $n\leq m_0+n_0$, 
 it follows from the definition of $m_0$, that $\lim_{l\rightarrow\infty}(f_{m_0+n_0-n}\circ\phi^{n+l}g_{n_0})(x_{0})= 0$, since $m_0+n_0-n<m_0$.
We choose a subsequence $\{l_{i_k}\}_{k\in\mathbb{N}}\subseteq\{l_i\}_{i\in\mathbb{N}}\subseteq\mathbb Z_+$, such 
that $l_{i_1}>L$ and $l_{i_{k+1}}-l_{i_{k}}>m_0+n_0+1$, for all $k\in\mathbb{N}$. 
We consider the sequence $\left\{U^{l_{i_k}}\chi\right\}_{k\in\mathbb{N}}$, where $\chi$ 
is the characteristic function of the singleton $\{\phi^{n_0}(x_0)\}$. To complete the proof we will prove that the 
sequence $\left\{M_{A,B}\left(U^{l_{i_{k}}}\chi \right)\right\}_{k\in\mathbb{N}}$ has no convergent subsequence. We estimate the quantity $\left\|M_{A,B}\left(U^{l_{i_{u}}}\chi\right)-M_{A,B}\left(U^{l_{i_{v}}}\chi\right)\right\|_{\mathcal A}$, for $u,v\in\mathbb{N}$, $u< v$.
\begin{eqnarray*}
\left\|M_{A,B}\left(U^{l_{i_{u}}}\chi\right)-M_{A,B}\left(U^{l_{i_{v}}}\chi\right)\right\|_{\mathcal A} & \ge &\\
\left\|E_{m_0+n_0+l_{i_{u}}}\left(M_{A,B}\left(U^{l_{i_{u}}}\chi\right)-M_{A,B}\left(U^{l_{i_{v}}}\chi\right)\right)\right\|_{C_0(X)} & = &\\
\left\|E_{m_0+n_0+l_{i_{u}}}\left(M_{A,B}\left(U^{l_{i_{u}}}\chi\right)\right)\right\|_{C_0(X)},
\end{eqnarray*}
since by the assumption $l_{i_{v}}-l_{i_{u}}>m_0+n_0+1$, the $(m_0+n_0+l_{i_{u}}){\mathrm {th}}$  Fourier coefficient of
$M_{A,B}\left(U^{l_{i_{v}}}\chi\right)$ is $0$.
We thus obtain 
\begin{eqnarray*}
\left\|M_{A,B}\left(U^{l_{i_{u}}}\chi\right)-M_{A,B}\left(U^{l_{i_{v}}}\chi\right)\right\|_{\mathcal A} & \ge &\\
\left\|\sum_{n=0}^{m_0+n_0}f_{m_0+n_0-n}\circ\phi^{n+l_{i_{u}}}\chi\circ\phi^{n}g_n\right\|_{C_0(X)}& \ge &\\
\left\|\sum_{n=0}^{m_0+n_0}f_{m_0+n_0-n}\circ\phi^{n+l_{i_{u}}}\chi\circ\phi^{n}g_ng_{n_0}\right\|_{C_0(X)}& \ge &\\
\left|\sum_{n=0}^{m_0+n_0}(f_{m_0+n_0-n}\circ\phi^{n+l_{i_{u}}}\chi\circ\phi^{n}g_ng_{n_0})(x_0)\right|& \ge &\\
|(f_{m_0}\circ\phi^{n_0+l_{i_{u}}}g^2_{n_0})(x_{0})|-\sum_{n=0, n\neq n_0}^{m_0+n_0}|(f_{m_0+n_0-n}\circ\phi^{n+l_{i_{u}}}\chi\circ\phi^{n}g_ng_{n_0})(x_0)|& \ge &\\
\left|(f_{m_0}\circ\phi^{n_0+l_{i_{u}}}g^2_{n_0})(x_{0})\right|-J(l_{i_{u}}) & > &\\\epsilon g_{n_0}(x_0),&&
\end{eqnarray*}
which concludes the proof. 
\end{proof}

\begin{lemma}\label{rl2}
Let $M_{A,B}:\mathcal A\rightarrow \mathcal A$ be a compact multiplication operator, where  $A,B\in \mathcal A_1$ and  $E_m(A)=f_m$, $E_m(B)=g_m$, for all $m\in\mathbb Z_+$.
Then, $(f_m\circ\phi^{n+l}g_n)(x)= 0$, for all  $m,n,l\in\mathbb Z_+$ and $x\in X_r$.

\end{lemma}
\begin{proof}
If  $x\in X_i\cap X_r$,  there exists a $k_0\in\mathbb Z_+$ such that $\phi^{k_0}(x)=x$.
It   follows from  Lemma \ref{cmo2} that  
 $\lim_{i\rightarrow\infty}(f_m\circ\phi^{n+l+ik_0}g_n)(x)= 0$, for $m, n, l\in\mathbb Z_+$. We note that $(f_m\circ\phi^{n+l+ik_0}g_n)(x)=
(f_m\circ\phi^{n+l}g_n)(x)$, for all $ i \in \mathbb Z_+$, and hence 
$(f_m\circ\phi^{n+l}g_n)(x)=0.$ If $x \in X_a\cap X_r$, the assertion follows from Lemma \ref{cmo1}.

\end{proof}


\begin{lemma}\label{cmo3}
Let $M_{A,B}:\mathcal A\rightarrow \mathcal A$ be a compact multiplication operator, where  
$A,B\in \mathcal A_1$ and  $E_m(A)=f_m$, $E_m(B)=g_m$, for all $m\in\mathbb Z_+$. 
Then, the sequence $\{f_m\circ\phi^{n+l}g_n\}_{l\in\mathbb{Z}_+}$ is pointwise equicontinuous, for all $m,n\in\mathbb Z_+$.
\end{lemma}
\begin{proof}
It follows from Lemma \ref{cmo1} that it is sufficient to prove that 
$\{f_m\circ\phi^{n+l}g_n\}_{l\in\mathbb{Z}_+}$ is pointwise equicontinuous on $X_{a, i}$.
We suppose that there exist some $n_0, m_0 \in\mathbb Z_+$ and a point $x_0\in X$, such that the sequence $\{f_{m_0}\circ\phi^{n_0+l}g_{n_0}\}_{l\in\mathbb{Z}_+}$ is not equicontinuous at $x_0$.
We note that $\left(f_{m_0}\circ\phi^{n_0+l}g_{n_0}\right)(x_0)=0$, for all $l\in\mathbb{Z}_+$, 
by Lemma \ref{cmo1}. Therefore, there exist an $\epsilon>0$, 
a strictly increasing sequence $\{l_i\}_{i\in\mathbb{N}}\subset\mathbb Z_+$ and a  
 sequence $\{x_i\}_{i\in\mathbb{N}}\subseteq X$, such that $\lim_{i\rightarrow\infty} x_i= x_0$ and
$\left|(f_{m_0}\circ\phi^{n_0+l_i}g_{n_0})(x_{i})\right|>\epsilon$, for all $\; i\in\mathbb{N}.$ We note that the inclusion $\{x_i\}_{i\in\mathbb{N}}\subseteq X_i$ holds by Lemma \ref{cmo1}.
Furthermore, we may assume that $l_{i+1}-l_i>m_0+n_0+1$, for all $i\in\mathbb{N}$. It follows from 
Lemma \ref{rl2}  that the elements $\{x_i\}_{i\in\mathbb{N}}$ are not  periodic. Therefore, if $m_0+n_0>0$, 
we have  that $\phi^{n_0}(x_i)\neq \phi^{n}(x_i)$, for all $n\in\{0,\dots,m_0+n_0\}\smallsetminus\{n_0\}$ and $i\in\mathbb{N}$.
We consider the sequence $\left\{U^{l_{i}}\chi_i\right\}_{i\in\mathbb{N}}$, where $\chi_i$ is the characteristic function of the set 
$\{\phi^{n_0}(x_i)\}$. Then, for $u< v$ we obtain,
\begin{eqnarray*}
\left\|M_{A,B}\left(U^{l_{u}}\chi_{u}\right)-M_{A,B}\left(U^{l_{v}}\chi_v\right)\right\|_{\mathcal A} & \ge &\\
\left\|E_{m_0+n_0+l_{u}}\left(M_{A,B}\left(U^{l_{u}}\chi_u\right)-M_{A,B}\left(U^{l_{v}}\chi_v\right)\right)\right\|_{C_0(X)} & = &\\
\left\|E_{m_0+n_0+l_{u}}\left(M_{A,B}\left(U^{l_{u}}\chi_u\right)\right)\right\|_{C_0(X)} & \ge &\\
\left|E_{m_0+n_0+l_{u}}\left(M_{A,B}\left(U^{l_{u}}\chi_u\right)\right)(x_{u})\right| & = &\\
\left|\sum_{n=0}^{m_0+n_0}(f_{m_0+n_0-n}\circ\phi^{n+l_{u}}\chi_{u} \circ \phi^{n}g_n)(x_{u})\right|& = &\\
\left|(f_{m_0}\circ\phi^{n_0+l_{u}}g_{n_0})(x_{u})\right| & > &\epsilon,
\end{eqnarray*}
since $x_u$ is not periodic.
Therefore, the sequence $\left\{M_{A,B}\left(U^{l_{i}}\chi_i\right)\right\}_{k\in\mathbb{N}}$ has 
no Cauchy subsequence.
\end{proof}


\begin{proposition}\label{compactelements}
Let  $m, n \in \mathbb Z_+$ and  $A=U^mf,\;B=U^ng\in\mathcal A_1$.  Then, 
the multiplication operator $M_{A,B}:\mathcal A\rightarrow \mathcal A$ is  compact if and only if the following assertions are valid.
\begin{enumerate}
\item $(f\circ\phi^{n+l}g)(X_a)=\{0\}$, for all $l\in\mathbb Z_+$,
\item $\lim_{l\rightarrow\infty}(f\circ\phi^{n+l}g)(x)= 0$, for all $x\in X_i$,
\item The sequence $\{f\circ\phi^{n+l}g\}_{l\in\mathbb{Z}_+}$ is pointwise equicontinuous.
\end{enumerate}
\end{proposition}
\begin{proof}
The forward direction is immediate by Lemmas \ref{cmo1}, \ref{cmo2} and \ref{cmo3}. We will show the opposite direction.

We will divide the proof in three steps:

\vspace{0.5em}
\textbf{1st step}

In this step we construct an approximation of $M_{A,B}$ by multiplication operators $M_{A_k, B_k}$
with the property that the Fourier coefficients
of $A_k, B_k$ are compactly supported.

 We define the following sets, for  $h\in C_0(X)$ and $k\in\mathbb N$.
 $$D_{k}(h)=\left\{x\in X:|h(x)|\ge\frac{1}{k}\right\},\ 
 U_{k}(h)=\left\{x\in X:|h(x)|>\frac{2}{3k}\right\}.$$ It is obvious that $D_{k}(h)\subseteq U_{k}(h)\subseteq D_{2k}(h)$ and that the set $\overline{U_{k}(h)}$ is compact.
 If $\|f\|\|g\|=0$, the proof is trivial. Otherwise, we choose a natural number  $k_s$, such that $\frac{1}{k_s}<\min\{\|f\|,\|g\|\}$. By Urysohn's lemma, 
 there are norm one functions $v_{f_k}$ and $v_{g_k}$ in $C_0(X)$, for $k>k_s$,  such that
\begin{eqnarray*}
v_{f_k}(x) =  \left\{\begin{tabular}{l} $1$ , $x\in D_k(f)$ \\ $0$ , $x\in X\setminus U_k(f)$\end{tabular}\right. 
&\text{and} &
v_{g_k}(x) =  \left\{\begin{tabular}{l} $1$ , $x\in D_k(g)$ \\ $0$ , $x\in X\setminus U_k(g)$\end{tabular}\right..
\end{eqnarray*}
We define the  functions
$f_{k}=v_{f_k}f ,\;   g_{k}=v_{g_k}g.$
It is immediate that $\|f-f_{k}\|\leq\frac{2}{k}$ and $\|g-g_{k}\|\leq\frac{2}{k}$. It follows that, if $A_k=U^mf_{k}$ 
and $B_k=U^ng_{k}$, we have that $\|A-A_k\|<\frac{2}{k}$ and $\|B-B_k\|<\frac{2}{k}$. Then, we can see that
\begin{equation*}
\sup_{T\in\mathcal A_1}\|M_{A,B}(T)-M_{A_k,B_k}(T)\|<\frac{4}{k}.
\end{equation*}
Hence, to prove that  $M_{A,B}$ is compact, 
it suffices to show that there exists a natural number $k_0$ such that  $M_{A_k,B_k}$ is compact, for all $k>k_0$.

\vspace{0.5em}
\textbf{2nd step}

1st case

 Firstly, we assume  that $\overline{U_{k'}(g)}\cap X_{a,i}\neq\emptyset$, 
for some $k' \in \mathbb Z_+$.
It follows that $\overline{U_{k}(g)}\cap X_{a,i}\neq\emptyset$, for all $k>k'$. 
Let $k_0=\max\{k_s,k'\}$ and $k>k_0$. If $\tilde x\in \overline{U_{k}(g)}\cap X_{a,i}$, there exists an open neighbourhood 
$V_{\tilde x}$ of $\tilde x$, such that $|g(x)|>\frac{1}{2k}$, for all $x\in V_{\tilde x}$.
Furthermore, $\tilde x$ is an accumulation point which in turn means that 
$(f\circ\phi^{n+l}g)(\tilde x)=0$, for all $l\in\mathbb{Z}_+$. Moreover, we recall that the family
$\{f\circ\phi^{n+l}g\}_{l\in\mathbb{Z}_+}$ is equicontinuous at $\tilde x$. 
Therefore, there exists an open neighbourhood $V_{\tilde x}'$ of $\tilde x$, such that 
$|(f\circ\phi^{n+l}g)(x)|<\frac{1}{4k^2}$, for all $x\in V_{\tilde x}'$ and $l\in\mathbb{Z}_+$. 
We set $W_{\tilde x}=V_{\tilde x}\cap V_{\tilde x}'$. It follows that
\begin{eqnarray*}
\frac{1}{4k^2} & > & |(f\circ\phi^{n+l}g)(x)|\\
&=& |f\circ\phi^{n+l}(x)||g(x)|\\
&\ge& |f\circ\phi^{n+l}(x)|\frac{1}{2k},
\end{eqnarray*}
and therefore, $|f\circ\phi^{n+l}(x)|\leq\frac{1}{2k}$, for all $x\in W_{\tilde x}$
and $l\in\mathbb{Z}_+$. We denote by $\cup_{\tilde x} W_{\tilde x}$ the set $\cup\{W_{\tilde{x}}\;|\; \tilde{x}\in\overline{U_k(g)}\cap X_{a,i}\}$. It follows that 
$\phi^{n+l}\left(\cup_{\tilde x} W_{\tilde x}\right)\subseteq X\setminus \overline{U_k(f)}$, for all $l\in\mathbb{Z}_+$ and hence $(f_{k}\circ\phi^{n+l}g_{k})(\cup_{\tilde x}W_{\tilde x})=\{0\}$.
Moreover, the set $\overline{U_{k}(g)}\setminus(\cup_{\tilde x}W_{\tilde x})$ is compact and 
$\left(\overline{U_{k}(g)}\setminus(\cup_{\tilde x}W_{\tilde x})\right)\cap X_{a,i}=\emptyset$, 
which in turn implies that the set $(\overline{U_{k}(g)}\setminus(\cup_{\tilde x}W_{\tilde x}))\cap X_i$ is finite.
We denote by $I_k$   the set $(\overline{U_{k}(g)}\setminus(\cup_{\tilde x}W_{\tilde x}))\cap X_i$ and 
by $\chi_{I_k}$ the  characteristic function of $I_k$.
We set $\tilde g_{k}=g_k\chi_{I_k}$.
We note that $(f\circ\phi^{n+l}g)(X_a)=\{0\}$, for all  $l\in\mathbb Z_+$ and hence  
$(f_k\circ\phi^{n+l}g_k)(X_a)=\{0\}$, for all  $l\in\mathbb Z_+$. Furthermore, 
we have proved that $(f_{k}\circ\phi^{n+l}g_{k})(\cup_{\tilde x}W_{\tilde x})=\{0\}$. Since the function $g_{k}$ is supported in
$U_k(g)$, we conclude that
\begin{equation*}
f_{k}\circ\phi^{n+l}g_{k}=f_{k}\circ\phi^{n+l}\tilde g_{k},
\end{equation*}
for all  $l\in\mathbb Z_+$.

2nd case

We assume now that 
$ \overline{U_{k}(g)}\cap X_{a,i}=\emptyset,\forall k\in\mathbb N$. 
 The set $\overline{U_{k}(g)}\cap X_i$ is finite, since the set $\overline{U_{k}(g)}$ is compact and $\overline{U_{k}(g)}\cap X_{a,i}=\emptyset$.
We   denote by the same letter as in the 1st case, $I_k$   the set $\overline{U_{k}(g)}\cap X_i$ and 
by $\chi_{I_k}$ the  characteristic function of $I_k$.
We set $\tilde g_{k}=g_k\chi_{I_k}$.
We note that $(f\circ\phi^{n+l}g)(X_a)=\{0\}$, for all  $l\in\mathbb Z_+$ and hence  
$(f_k\circ\phi^{n+l}g_k)(X_a)=\{0\}$, for all  $l\in\mathbb Z_+$.  Since $g_{k}$ is supported in
$U_k(g)$, we conclude that
\begin{equation*}
f_{k}\circ\phi^{n+l}g_{k}=f_{k}\circ\phi^{n+l}\tilde g_{k},
\end{equation*}
for all  $l\in\mathbb Z_+$. 
It follows that $M_{A_k,B_k}= M_{A_k,\tilde B_k}$, where 
$\tilde B_k=U^n\tilde g_{k}$.

\vspace{0.5em}
\textbf{3rd step}

It follows that 
$\lim_{l\rightarrow\infty}(f_k\circ\phi^{n+l}g_k)(y)=0$, for all $y \in I_k$, by assumption.
We observe that $\lim_{l\rightarrow\infty}(f_k\circ\phi^{n+l})(y)=0$, since $g_{k}(y)\neq0$, which follows from the inclusion $I_k\subseteq \overline{U_k(g_k)}$.
Therefore, there exists an $L_{0}\in\mathbb N$, such that 
$\phi^{n+l}\left(y\right)\in  (X\setminus \overline{U_{k}(f)})$,   for all $l\ge L_{0}$ and for all $y \in I_k$, since the set $I_k$ is finite.
Since $f_k$ is supported in $\overline{U_{k}(f)}$, we obtain  that $f_{k}\circ\phi^{n+l}\tilde g_{k}=0$, for all   $l\ge L_0$. 
It follows that 
\begin{equation*}
M_{A_k,\tilde B_k}(U^lh)=0,
\end{equation*}
for all $l \ge L_0$ and $h \in C_0(X)$.
Since  $\tilde g_{k}$ has   finite support, the operator  
$
M_{A_k,\tilde B_k},$ is a finite rank operator and hence compact.

\end{proof}

Let $A$ be an element of the semicrossed product $\mathcal{A}$.
We consider the sequence $\{U^nf_n\}_{n\in\mathbb{Z}_+}\subseteq\mathcal A$, where $f_n=E_n(A)$, for $n\in\mathbb{Z}_+$. 
We note that the series $\sum_{n\in\mathbb{Z}_+} U^nf_n$ does not converge to $A$ in general. 
The $k$th arithmetic mean of $A$ is defined to be the element $A_k=\frac{1}{k+1}\sum_{l=0}^k S_l$, 
where $S_l=\sum_{n=0}^l U^nf_n$. Then, the sequence $\{A_k\}_{k\in\mathbb{Z}_+}$ is norm convergent to $A$ \cite[p. 524]{peters}.

\begin{theorem} \label{main2}
Let $A,B\in \mathcal A$ and  $E_m(A)=f_m\in C_0(X)$, $E_m(B)=g_m\in C_0(X)$, for all $m\in\mathbb Z_+$. 
The following statements are equivalent.
\begin{enumerate}
\item The multiplication operator $M_{A,B}:\mathcal A\rightarrow \mathcal A$ is  compact.
\item The following assertions are valid, for all $m,n\in\mathbb Z_+$.
\begin{enumerate}
\item $(f_m\circ\phi^{n+l}g_n)(X_a)=\{0\}$, for all $l\in\mathbb Z_+$.
\item $\lim_{l\rightarrow\infty}(f_m\circ\phi^{n+l}g_n)(x)= 0$, for all $x\in X_i$.
\item The sequence $\{f_m\circ\phi^{n+l}g_n\}_{l\in\mathbb{Z_+}}$ is pointwise equicontinuous.
\end{enumerate}
\end{enumerate}
\end{theorem}
\begin{proof}
It is sufficient to prove the theorem for $A, B \in \mathcal A_1$.

The condition (1) implies the condition (2) by Lemmas \ref{cmo1}, \ref{cmo2} and \ref{cmo3}. We will show the opposite direction.

If $A=\sum_{m=0}^pU^mf_m$ and $B=\sum_{n=0}^qU^ng_n$,  for $ p, q \in \mathbb Z_+$, we have 
 $$M_{A, B}=\sum_{m=0}^p\sum_{n=0}^qM_{U^mf_m, U^ng_n}$$  and the assertion folows 
from Proposition \ref{compactelements}.

If $A, B \in \mathcal A$  and 
 $k\in\mathbb{Z}_+$, we denote by 
$A_k$ and $B_k$ the $k$th arithmetic mean of $A$ and $B$ respectively. Since the Fourier coefficients of $A$ and $B$ 
satisfy the condition
$(2)$, the Fourier coefficients of $A_k$  and $B_k$ satisfy the condition $(2)$ as well.
Thus,  the operator   $M_{A_k, B_k}$ 
is compact, for all $k \in \mathbb Z_+$.
The operator $M_{A, B}$ is the  norm limit of the sequence  $\{M_{A_k, B_k}\}_{k\in\mathbb{Z}_+}$ and hence it is compact.

\end{proof}

As a corollary of the above theorem, we obtain the following characterization of the compact elements of the algebra $\mathcal A$.


\begin{corollary} \label{compelements}
Let $A\in\mathcal{A}$ and $E_m(A)=f_m\in C_0(X)$, for all $m\in Z_+$. Then,
 $A$   is compact element of $\mathcal{A}$, if and only if the following conditions are satisfied, for all $m,n\in\mathbb{Z}_+$.
\begin{enumerate}
\item $(f_m\circ\phi^{n+l}f_n)(X_a)=\{0\}$, for all $l\in\mathbb Z_+$.
\item $f_m(X_r)=\{0\}$.
\item The sequence $\{f_m\circ\phi^{n+l}f_n\}_{l\in\mathbb{Z}_+}$ is pointwise equicontinuous.
\end{enumerate}
\end{corollary}
\begin{proof}

It is sufficient to prove the corollary for $A \in \mathcal A_1$.

 Firstly, we will show the forward direction. The conditions (1) and (3) are satisfied by Theorem \ref{main2}.
 Let $x$ be a recurrent point. Then,
there exists a strictly increasing sequence $\{l_i\}_{i\in\mathbb{N}}\subseteq\mathbb Z_+$, such that 
$\lim_{i\rightarrow\infty}\phi^{m+l_i}( x)= x$, which implies that
$\lim_{i\rightarrow\infty}(f_m\circ\phi^{m+l_i}f_m)(x)=f_m^2(x). $ It follows from Lemma \ref{rl2} that $f_m(x)=0$.

Now, we will show the opposite direction. In view of Theorem \ref{main2}, it suffices to prove that 
 $\lim_{l\rightarrow\infty}(f_m\circ\phi^{n+l}f_n)(x)=0$, for all $m,n\in\mathbb Z_+$ and $x\in X_i$.
Let $m_0,n_0\in\mathbb Z_+$ and $x_0\in X_i$, such that $\lim_{l\rightarrow\infty}(f_{m_0}\circ\phi^{n_0+l}f_{n_0})(x_0)\neq 0$. 
Then, there exists a natural number $k_0$ and a strictly increasing sequence $\{l_i\}_{i\in\mathbb{N}}\subseteq\mathbb Z_+$, such 
that $|(f_{m_0}\circ\phi^{n_0+l_i}f_{n_0})(x_0)|\ge \frac{1}{k_0}$, for all $i\in\mathbb N$.  
Hence, $|(f_{m_0}\circ\phi^{n_0+l_i})(x_0)|\ge \frac{1}{k_0}$, for all  $i\in\mathbb N$,  which implies that 
$\phi^{n_0+l_i}(x_0)\in D_{k_0}(f_{m_0})=\{x\in X : f_{m_0}(x)\geq\frac{1}{k_0}\}$,  for all  $i\in\mathbb N$. By the condition (2), we obtain that  $x_0\in X_i\setminus X_r$ 
and hence, $\phi^{n_0+l_i}(x_0)\neq \phi^{n_0+l_j}(x_0)$, for $i\neq j$. Moreover, the set $D_{k_0}(f_{m_0})$ is compact and 
hence, there exists a point $\tilde x\in X_{a,i}\cap D_{k_0}(f_{m_0})$ and a subsequence 
$\{l_{i_j}\}_{j\in\mathbb{N}}$, such that $\lim_{j\rightarrow\infty}\phi^{n_0+l_{i_j}}(x_0)= \tilde x$. 
Let $W_{\tilde x}$ be  an open neighbourhood of $\tilde x$,
 such that $|f_{m_0}(x)|>\frac{1}{2k_0}$, for all $x\in W_{\tilde x}$.
By the condition (1), we have that $(f_{m_0}\circ\phi^{n_0+l}f_{m_0})(\tilde x)=0$, for all $l \in \mathbb N$ and 
by  (c) we have that the sequence $\{f_{m_0}\circ\phi^{n_0+l}f_{m_0}\}_{l\in\mathbb{Z}_+}$ is equicontinuous at $\tilde x$. 
Therefore, there is an open neighbourhood $W_{\tilde x}'$ of $\tilde x$, such
that $|(f_{m_0}\circ\phi^{n_0+l}f_{m_0})(x)|<\frac{1}{5k_0^2}$, for all $l \in \mathbb Z_+$ and $x\in W_{\tilde x}'$.
Let $W=W_{\tilde x}\cap W_{\tilde x}'$ and $j_1,j_2\in\mathbb N$ be 
such that $\phi^{n_0+l_{i_{j_1}}}(x_0),\phi^{n_0+l_{i_{j_2}}}(x_0)\in W$ and $l_{i_{j_2}}>l_{i_{j_1}}+n_0$. Then,
\begin{eqnarray*}
\frac{1}{5k_0^2}&\ge&|(f_{m_0}\circ\phi^{l_{i_{j_2}}-l_{i_{j_1}}}f_{m_0})(\phi^{n_0+l_{i_{j_1}}}(x_0))|\\
&=&|f_{m_0}(\phi^{n_0+l_{i_{j_2}}}(x_0))||f_{m_0}(\phi^{n_0+l_{i_{j_1}}}(x_0))|\ge\frac{1}{4k_0^2},
\end{eqnarray*}
which is a contradiction.
\end{proof}

Let us see now how Theorem \ref{main2} applies to two special cases. If $X$ is a discrete space or it has no isolated points, we obtain the following characterizations.

\begin{corollary}
Let $X$ be a discrete space, $A,B\in \mathcal A$ and  $E_m(A)=f_m,E_m(B)=g_m$, for all $m\in\mathbb Z_+$.
Then, the following are equivalent.
\begin{enumerate}
\item The multiplication operator $M_{A,B}:\mathcal A\rightarrow \mathcal A$ is  compact.
\item $(f_m\circ\phi^{n+l}g_n)(X_r)=\{0\}$, for all $m,n,l\in\mathbb Z_+$.
\item
$\lim_{l\rightarrow \infty}(f_m\circ\phi^{n+l}g_n)(x)= 0$, for all $x\in X .$

\end{enumerate}
\end{corollary}
\begin{proof}
We will show the implication $(2)\Rightarrow (3)$.
Assume  that there exist  $m, n \in \mathbb Z_+$  and 
$x\in X\setminus X_r$, such that $\lim_{l\rightarrow\infty}(f_m\circ\phi^{n+l}g_n)(x)\neq0$. 
Then, there exist an  $\epsilon >0$ and a strictly increasing  sequence $\{l_i\}_{i\in\mathbb{N}}\subseteq\mathbb Z_+$, such that
$|(f_m\circ\phi^{n+l_i}g_n)( x)|\ge\epsilon$, for all $i\in\mathbb N$.
Moreover, $ x\notin X_r$ and therefore  
$\phi^{n+l_i}( x)\neq \phi^{n+l_j}( x)$, for $i\neq j$. This is a 
contradiction, since $f_m \in C_0(X)$.

The implication $(3)\Rightarrow (1)$ follows from Theorem  \ref{main2} and 
the implication $(1)\Rightarrow (2)$ follows from Lemma  \ref{rl2}.
\end{proof}

\begin{corollary}
Let $X$ be a space without isolated points, $A,B\in \mathcal A$ and  $E_m(A)=f_m,E_m(B)=g_m$, for all $m\in\mathbb Z_+$.
Then, the following are equivalent.
\begin{enumerate}
\item The multiplication operator $M_{A,B}:\mathcal A\rightarrow \mathcal A$ is  compact.
\item $(f_m\circ\phi^{n+l}g_n)(X)=\{0\}$ for all $l,m,n\in\mathbb N$.
\item $M_{A,B}=0$.
\end{enumerate}
\end{corollary}
\begin{proof}
 It follows immediately from Theorem \ref{main2}.
\end{proof}

The next example shows that the equicontinuity condition can not be omitted in general.

\begin{example}
We consider the dynamical system $(X,\phi)$ where
\begin{eqnarray*}
X=\{0\}\cup\{x_n\}_{n\in\mathbb Z}\cup \{2\}, &  & x_n=\left\{\begin{matrix}\frac{1}{|n|+1} & \hspace{-1em}, & n<0 \\ 2-\frac{1}{n+1} & \hspace{-1em}, & n\ge 0\end{matrix}\right.
\end{eqnarray*}
and $\phi$ is the homeomorphism
\begin{eqnarray*}
\phi(0)=0 , & \phi(x_n)=x_{n-1} , & \phi(2)=2.
\end{eqnarray*}
We define the  monomials $A=U^1f$ and $B=U^1g$ of the semicrossed product $\mathcal{A}$ by the following formulae.
\begin{eqnarray*}
f(x)=\left\{\begin{matrix} 1, &  &\hspace{-0.8em} x=1 \\ 0, &  &\hspace{-1.5em} \text{else}\end{matrix}\right. & \text{and} & g(x)=\left\{\begin{matrix} 1, &  &\hspace{-1em} x> 1 \\ 0, &  &\hspace{-1em} x\leq 1\end{matrix}\right..
\end{eqnarray*}
We observe that, $(f\circ\phi^{1+l}g)(X_a)=\{0\}$, for all $l\in\mathbb Z_+$ and 
$\lim_{l\rightarrow\infty}(f\circ\phi^{1+l}g)(x)=0$, for all $x\in X_i$.
However, the sequence $\{f\circ\phi^{1+l}g\}_{l\in\mathbb{Z}_+}$ is not equicontinuous at $x=2$ and the multiplication operator $M_{A,B}:\mathcal{A}\rightarrow \mathcal{A}$ is not compact. 
\end{example}



In the following theorem, we characterize the ideal generated by the set of 
compact elements of the semicrossed product $\mathcal{A}$. Recall that  $Y\subseteq  X$ is called wandering if the sets 
$\phi^{-1}(Y), \phi^{-2}(Y), ...$ are pairwise disjoint. If $\phi$ is a homeomorphism, this condition is equivalent to
the condition that $\phi^m(Y)\cap \phi^n(Y)=\emptyset$ for all $m, n \in \mathbb Z_+, m\neq n$. 
A point $x \in X$ is called wandering if it possesses an open wandering neighborhood. Otherwise it is called non wandering. If $x$ 
is a non wandering point and $W$ is an open neighbourhood of $x$, then  there exists $m \in \mathbb N$, such 
that $W\cap \phi^m(W)\neq \emptyset$. Note that we may assume that $m$ is arbitrarily large \cite[p. 129]{kh}.
We will denote by $X_w$ the set of wandering points of $X$. It is clear that $X_w$  is 
the the union of all open wandering subsets of $X$.

\begin{theorem}\label{cmo31}
 The ideal generated by the compact elements of $\mathcal{A}$ is 
 the set 
 $$
 \{A\in\mathcal{A}\; |\; E_n(A)(X\setminus X_w)=\{0\},\; \forall n\in\mathbb{Z_+}\ \text{and } E_0(A)(X_a)=\{0\}\}.$$
\end{theorem}

\begin{proof}
Let $A$ be a compact element of $\mathcal{A}$.
We will show that  
$E_n(A)(X\setminus X_w)=\{0\}$, for all $n\in\mathbb{Z_+}$ and $
E_0(A)(X_a)=\{0\}$.
 Let us  denote by $f_n$ the $n$th Fourier coefficient, $E_n(A)$, of $A$, for all $n\in\mathbb{Z}_+$.  
We suppose that there exists an $n_0\in\mathbb Z_+$ and a point $x_0\in X\setminus X_{w}$, such that $f_{n_0}(x_0)\neq 0$. 
Without loss of generality, we assume that $|f_{n_0}(x_0)|=1$.
We note that by Corollary \ref{compelements}, $x_0$ cannot be isolated,  since  a non wandering  isolated 
point is periodic and hence recurrent.
Thus, $x_0$ is an accumulation point, and we have $(f_{n_0}\circ\phi^{n_0+l}f_{n_0})(x_0)=0$, 
for all $l\in\mathbb Z_+$. Let $0<\epsilon<1/4$. The sequence $\{f_{n_0}\circ\phi^{n_0+l}f_{n_0}\}_{l\in\mathbb{Z}}$ 
is equicontinuous at 
$x_0$ and therefore, there exists an open neighborhood $U$ of $x_0$, such that $|(f_{n_0}\circ\phi^{n_0+l}f_{n_0})(x)|<\epsilon$, for all $x\in U$ and $l\in\mathbb Z_+$.
There exists an open neighborhood $V$ of $x_0$, such that $|f_{n_0}(x)|>1-\epsilon$, for all $x\in V$.
Let $W=U\cap V$. Since $x_0$ is a non wandering point, there exists an $l_0\in\mathbb Z_+$ such that $l_0\ge n_0$ 
 and $W\cap\phi^{l_0}(W)\neq\emptyset$. 
Let $y\in W$, such that $\phi^{l_0}(y)\in W$. Then,
\begin{eqnarray*}
\epsilon>|(f_n\circ\phi^{l_0}f_n)(y)|=|f_n(\phi^{l_0}(y))||f_n(y)|>(1-\epsilon)^2
\end{eqnarray*}
which is a contradiction. Furthermore, it is evident from Corollary \ref{compelements}, that $f_0(X_a)=\{0\}$. 
We observe that the conditions $f_n(X\setminus X_w)=\{0\}$,  for all $n\in\mathbb{Z}_+$ and $f_0(X_a)=\{0\}$, 
we have already proved for a compact element $A$, are satisfied by the elements of the ideal generated by the 
compact elements as well.

To complete the proof, we will prove that if $A \in \mathcal A$ satisfies 
$E_n(A)(X\setminus X_w)=\{0\}$, for all $n\in\mathbb{Z_+}$\ \text{and } $E_0(A)(X_a)=\{0\}$, then $A $  belongs to
the ideal generated by the 
compact elements.
We  denote by $f_n$ the $n$th Fourier coefficient $E_n(A)$ of $A$, for all $n\in\mathbb{Z}_+$.  
It is sufficient to show that  $U^nf_n$ 
belongs to the ideal generated by the 
compact elements, $\forall n\in\mathbb{Z_+}$.

First,  we  consider  the case $n>0$.
Let $\mathcal C$ be the set
 $\{h\in C_0(X): \D(h)=h^{-1}(\mathbb C\smallsetminus\{0\}) \textrm{ is wandering}\}$. 
 Let $h \in \mathcal C$.
 Since $\D(h)$ is wandering and $n>0$, the functions 
 $(h\circ\phi^{n+l}h)$ are identically $0$, for every $l\in\mathbb Z_+$. Hence, $U^nh$ satisfies  the conditions (a)  and (c)
 of Corollary \ref{compelements}. Since $X_r$  is contained in $X\setminus X_w$  condition (c)  is also satisfied and $U^nh$  
 is a compact element of $\mathcal A$.
 The norm closed algebra generated by $\mathcal C$, 
 is the  ideal
$\{f\in C_0(X): f(x)=0,\;  \forall x\in X\setminus X_w\}$ of $C_0(X)$. In particular, $f_n$ belongs to this algebra.
We conclude that $U^nf_n $  belongs to the ideal generated by the 
compact elements.

We consider now the case $n=0$. The proof is the similar to the proof in the case $n>0$, considering the set
$\mathcal F=\{h\in C_0(X):h(X_a)=\{0\}, \D(h)=h^{-1}(\mathbb C\smallsetminus\{0\}) \textrm{ is wandering}\}$ instead of $\mathcal C$.

\end{proof}

\bibliographystyle{amsplain}

\end{document}